\documentclass[11 pt,reqno]{amsart}

\usepackage{amsmath,amssymb,mathrsfs}
\usepackage{amsthm}  

\usepackage{cite} 

\newcommand{\B}{\mathbb{B}}
\newcommand{\C}{\mathbb{C}}
\newcommand{\E}{\mathbb{E}}

\newcommand{\R}{\mathbb{R}}

\newcommand{\Rn}{\mathbb{R}^n}

\newcommand{\mc}[1]{\mathcal{#1}}
\newcommand{\mr}[1]{\mathrm{#1}}

\newcommand{\abs}[1]{\left|{#1}\right|}
\newcommand{\nrm}[1]{\left\|{#1}\right\|}
\newcommand{\set}[1]{\left\{#1\right\}\relax}
\newcommand{\scal}[1]{\left\langle\relax #1 \relax\right\rangle}
\newcommand{\qtxtq}[1]{\quad \text{#1}\quad}

\newtheorem{thm}{Theorem}[section]
\newtheorem{defn}[thm]{Definition}
\newtheorem{prop}[thm]{Proposition}
\newtheorem{lem}[thm]{Lemma}
\newtheorem{cor}[thm]{Corollary}
\newtheorem{example}[thm]{Example}
\newtheorem{rem}[thm]{Remark}

\author{S. Loukili and M. Maslouhi}
\address{%
	S.Loukili\\%
	Ibn Tofail University. Kenitra 14000. Morocco.
}
\email{%
	loukili.sl@gmail.com
}

\address{%
	M.Maslouhi\\%
	National School of Applied Sciences\\%
	Ibn Tofail University. Kenitra 14000. Morocco.
}

\email{%
	mostafa.maslouhi@uit.ac.ma
}

\title{A new probabilistic model for optimal  frames in erasure's recovery}

\begin{document}
	
	\begin{abstract}
		In this paper we introduce a new probabilistic model for modeling occurrence of erasures in a data transmission. This new model uses a sequence of Bernoulli random variables to model the channels of transmission and Parseval frames to encode transmitted data. Our model gives insights on the probabilistic properties of the channels and allows us to find optimal Parseval frames that minimize the lost of data in the worst one erasures cases. We show also that compared to existing models \cite{holmes2004optimal,casazza2003equal,leng2013probability,li2018frame}, our optimal Parseval frames give better performance for recovering transmitted data.
	\end{abstract}
	
	\maketitle

	
	
	
	\section{Introduction and preliminaries}\label{sec:intro}
	Transmission of signals using frames, a redundant set of vectors in a Hilbert space (usually $\R^n$ or $\C^n$), are advantageous over orthonormal or Riesz bases due to their redundancy property which reduces losses and errors of reconstructed signals when some coefficients are
	lost in the transmission process. A lot of works is oriented towards  searching the best frames, according to some chosen criteria, optimizing the errors occurred during the recovery of transmitted signals. See \cite{casazza2003equal,holmes2004optimal,bodmann2005,leng2011optimal,leng2011optimalprob,leng2013probability,li2018frame} and the references therein.
	
	The source information is viewed as a vector $x\in\mathbb{K}^n$ ($\mathbb{K}=\mathbb{R}$ or $\mathbb{C}$), which is decomposed, via $m\geq n$ expansion coefficients with respect to some fixed frame.  These $m$ scalars are then sent over the networks. Due to unpredictable communication losses, the receiver may not receive the whole $m$ packets. During the last decades, one big problem for researchers is to construct optimal frames that minimize the recovering error of the original signals. In particular,   it has been shown  \cite{casazza2003equal,holmes2004optimal,bodmann2005} that uniform tight frames are optimal for 1-erasure and the equiangular frames are optimal for 2-erasures. ( See below for appropriate definitions). 
	
	We point out here that a probabilistic model related to erasure recovery was introduced in \cite{leng2011optimalprob,leng2013probability,li2018frame}. In this model, the authors  assigned to each of the $m$ channels a probability of loss $p_i$, $i=1, \dots, m$. Next, they consider a new (transformed) probability distribution:  $q_i=F(p_1,\dots,p_m)$, $i=1,\dots,m$.  
	
	In our model, introduced later, we consider only the original probability distribution $(p_i)_{i=1}^m$, and we obtain even better results than  \cite{casazza2003equal,holmes2004optimal,leng2013probability,li2018frame} as it is proved in Section \ref{sec:comparison}.
	
	The remaining of this section recalls some notations and basics of frame theory. 
	In Section \ref{sec:a_prob_model}, we present some existed probabilistic models dealing with erasure recovery and using frame expansion theory.
	
	Our revisited probabilistic model is set up in Section \ref{sec:revising_prob_model}, where we  establish also our main result characterizing  optimal frames for the minimization problem related to the recovery error of the transmitted data. 
	
	Section \ref{sec:comparison} is devoted to compare our results to those obtained in  \cite{casazza2003equal,holmes2004optimal,leng2013probability,li2018frame}.
	
	Throughout this paper, we will refer by (CM) and (PM) to those models introduced in  \cite{casazza2003equal,holmes2004optimal} and 
	\cite{leng2013probability,li2018frame} respectively, while (RPM) will refer to our revisited model introduced later in Section \ref{sec:revising_prob_model}.
	
	A sequence $f:=(f_i)_{i=1}^m\subset \R^n$ is called a frame  if there exist constants $0<A\leq B<\infty$ such that 
	$$A\Vert x\Vert^2\leq\sum_{i=1}^m\vert\langle x,f_i\rangle\vert^2\leq B\Vert x\Vert^2,$$
	holds for every $x\in\R^n$. If $A=B$, the frame is called tight. When $A=B=1$, then it is said to be a Parseval frame. If all elements of the frame have the same norm, it is called   uniform frame. If $\vert\langle f_i,f_j\rangle\vert$ is a constant for all $i\neq j$, the frame is called   equiangular.
	
	In all the sequel, a frame of $m$ vectors in $\R^n$ will be called a $(m,n)$-frame and the set of all $(m,n)$-frames will be  denoted by $\mathcal{F}(m,n)$, while $\mathcal{P}(m,n)$ will denote its subsets consisting of Parseval frames.
	
	The frame analysis operator related to $f$ is denoted by $T:=T_f$ and given by
	$$Tx:=\left(\scal{x,f_i}\right)_{i=1}^m,\quad  x\in\R^n,$$
	and the frame synthesis operator is 
	given by $$T^{*}(c_1,\dots,c_m):= \sum_{i=1}^m c_if_i, \quad (c_1,\dots,c_m)\in\R^m.$$
	
	From now on,  we identify the operators $T$ and $T^*$, related to $f$,  with there associated matrices in the canonical basis. 
	
	The  operator $S:=T^{*}T$  is called the  frame operator  associated to the frame $f$. It is known that $S$ is a symmetric positive definite operator.
	
	The next result, known as the fundamental inequality, is a useful tool for our developments later. 
	\begin{prop}[ \cite{casazza2006physical}]\label{pro:fundam_ineq}
		Let $\mathcal{H}$ be an $n-$dimensional Hilbert space, and let $(a_i)_{i=1}^m$ be a  sequence of positive real numbers. Then, there exists a tight
		frame $( f_i)_{i=1}^m\subset\mathcal{H}$ with $\nrm{f_i}= a_i$ for all $i$, if and only if we have
		$$\max_{i=1,\dots,m}a_i^2\leq\frac{1}{n}\sum_{i=1}^ma_i^2.$$
	\end{prop}
	
	\section{Probabilistic model for erasures} 
	\label{sec:a_prob_model}
	In this section we present some of the basics of data erasure recovery using frame expansion theory along with a  probabilistic model.
	
	Let $f:=( f_i)_{i=1}^m$ be a Parseval frame in $\R^n$. 
	The data to be transmitted by means of the frame $f$ is decomposed into packets $x\in\mathbb{R}^m$, and each packet $x$ is sent, to some receiver, as a vector $y=Tx=(\scal{x,f_i})_{i=1}^m\in\mathbb{R}^m$, where $T$ is the analysis operator associated to $f$. 
	
	In the standard recovery model way, we let $\mathcal{R}$ denote the set of indexes corresponding to surviving coefficients of $y$. Thus, the recovered vector is given by $\tilde{x}=T^*Qy$, where  $Q$ stands for the $m\times m$ matrix  defined by
	\begin{equation}\label{eq:def_matrix_Q}
		\begin{cases}
			Q_{ii}=1 \quad i\in\mathcal{R},\\
			Q_{ij}=0\quad\mathrm{otherwise},
		\end{cases}
	\end{equation}
	and the reconstruction error is evaluated by 
	\begin{equation}\label{eq:def_error}
		x-\tilde{x}= (T^*DT)x,
	\end{equation}
	where $D:=I_m-Q$. 
	
	The main concern of optimizing  data erasure is to seek, within the set of $(m,n)$-frame $f$, the optimums of the quantity 
	\begin{equation}\label{eq:def_d_r}
		d_r(f):= \max_{D\in\mathcal{D}_r}\Vert T^*DT\Vert,
	\end{equation}
	where, for $1 \leq r\leq m$, $\mathcal{D}_r,$ denotes the set of $m\times m$ diagonal matrices with $r$ ones and $m-r$ zeros and the used norm here is the operator norm.
	See \cite{casazza2003equal,bodmann2005,holmes2004optimal} for details and results related to the optimization problem associated to \eqref{eq:def_d_r} in the cases $r=1,2$.
	
	To enhance the minimization problem related to \eqref{eq:def_d_r}, a probabilistic model of erasures was introduced in  \cite{leng2011optimalprob,leng2013probability,li2018frame}. More precisely, they associate to the  $i^\text{th}$ channel, $i=1,\dots,m$, a coefficient $p_i$,  representing the probability that erasure occurs in this channel, and in place of the matrix $Q$ defined in \eqref{eq:def_matrix_Q},  they used a new matrix
	\begin{equation}\label{eq:def_matrix_Q_p}
		\begin{cases}
			Q_{ii}=q_i \quad i\in\mathcal{R},\\
			Q_{ij}=0\quad\mathrm{otherwise},
		\end{cases}
	\end{equation}
	where the weights $(q_i)_{i=1}^m$ are given by 
	\begin{equation}\label{eq:def_prob_q_i}
		q_i:=\frac{m-1}{n}\frac{\sum_{k=1}^m p_k}{ \sum_{k=1}^m p_k-p_i}.
	\end{equation}
	See \cite{leng2011optimalprob,leng2013probability,li2018frame} for details.
	
	\section{A revisited probabilistic model} \label{sec:revising_prob_model}
	
	In this section, we introduce a new revisited probabilistic model by associating a sequence of independent Bernoulli random variables to the channels of transmission. More precisely, fix  a given frame $f=(f_i)_{i=1}^m$. By means of the frame $f$, each vector $x\in \Rn$ is decomposed as
	\begin{equation}\label{eq:repr_x}
		S x=\sum_{i=1}^{m} \scal{x,f_i}f_i,\quad x\in\R^n.
	\end{equation}
	According to our model, the recovered vector $\widetilde{x}$ satisfies  
	\begin{equation}\label{eq:def_x_tilde}
		S\widetilde{x}=\sum_{i=1}^{m}(1-X_i) \scal{x,f_i}f_i,\quad x\in\R^n,
	\end{equation}
	where  $X_i$, $i=1,\dots, m$ is a sequence of independent Bernoulli random variables associated to the $m$ channels such that 
	\begin{center}
		``$X_i=1$'' is the event ``Erasure occurred at the  $i^\text{th}$ channel",
	\end{center}
	and   
	\begin{equation}\label{eq:def_Bernoulli_distrib}
		\mathbb{P}(X_i=1)=p_i, \quad i=1,\dots, m,
	\end{equation}
	where the $p_i$'s are given.
	
	In all the sequel, we will refer to our revisited probabilistic model by (RPM).
	
	In our model (RPM),  we consider the ``random'' error related to a given frame $f$  by
	\begin{equation}\label{eq:def_random_error}
		\nrm{x-\tilde{x}}.
	\end{equation} 
	
	From \eqref{eq:repr_x} and \eqref{eq:def_x_tilde}, we have
	\begin{equation}\label{eq:expr_random_error}
		x-\tilde{x}=(S^{-1}T^*D_XT)x,
	\end{equation}where $D_X:=\mr{diag}(X_1,\dots,X_m)$ and $T:=T_f$, $S:=S_f$ are respectively the analysis and frame operators associated to the frame $f$. By consequence,
	$$	\nrm{x-\tilde{x}}\leq \nrm{(S^{-1}T^*D_XT)}\nrm{x},$$
	where  $\nrm{.}$ is the operator norm induced by the Euclidean norm in $\R^{n}$.
	
	Thus, the random error in \eqref{eq:def_random_error} is related to the random operator $\nrm{S^{-1}T^*D_XT}$.
	
	First of all, notice that for a given frame $f$, the frame $g:=S^{-1/2}f$ is  Parseval and  satisfies 
	$$\nrm{ S^{-1}T_f^*D_XT_f}=\nrm{ S^{-1/2}T_f^*D_XT_fS^{-1/2}}=\nrm{ T_g^*D_XT_g}.$$
	This allows us to restrict ourselves to the smallest set of Parseval frames and our main concern in this paper is to optimize the expression
	\begin{equation}\label{eq:exp_random_error}
		\mathbb{E}\left(\nrm{T^*D_XT}\right),
	\end{equation}
	over all the Parseval frames $f\in\mathcal{P}(m,n)$, which is yet a difficult problem. 
	
	To get around this difficulty, we will orient our developments in two different directions. In the first one, during the remaining of this section, we will establish several estimates of \eqref{eq:exp_random_error}. Next, see Section \ref{sec:worst_r_erasure} for details, following  
	\cite{casazza2003equal, holmes2004optimal, bodmann2005, leng2011optimalprob,  leng2013probability, li2018frame,    leng2011optimal} we will investigate, for $r=1$ and $r=2$, the worst r-erasure cases minimization problem, that is 
	\begin{equation}\label{eq:def_r_worst_case}
		\inf_{f\in \mc{P}(m,,n)}\left( \max_{J\in \mc{J}_r}\nrm{(f_{i})_{i\in J}(f_{i})_{i\in J}^*}\prod_{i\in J}p_i\prod_{i\notin J}(1-p_i)\right),
	\end{equation}
	where  $\mc{J}_r$ denotes the set of none empty subsets  of $\set{1,\dots,m}$ having exactly $r$ elements, $r\geq 1$, and Here,  $\left(f_{i}\right)_{i\in J}$ stands for the matrix whose columns are $f_{i}$, $i\in J$.  
	
	To see where \eqref{eq:def_r_worst_case} comes from, notice that for each none empty $J\subset \set{1,2,\dots,m}$, when erasures occur at the channels $\set{C_i,i\in J}$, the random operator
	$\nrm{T^*D_XT}$ takes the value $\nrm{(f_i)_{i\in J}(f_i)^*_{i\in J}}$ with probability
	$$\prod_{i\in J}p_i\left(\prod_{i\notin J}(1-p_i)\right).$$ 
	
	\bigskip
	
	Up to rename the channels, we  may, and  will from now on,  suppose that  
	$$p_1\leq p_2\leq \dots \leq p_m.$$	
	
	For a frame $f=(f_i)_{i=1}^m$ we have
	\begin{equation}\label{eq:oper_by_X_i}
		T^*D_XT=\sum_{i=1}^mX_i f_i\otimes f_i,
	\end{equation}   
	
	In the remaining of this section, we establish some probabilistic properties  which give a threshold of the  mean  $\E(\nrm{T^*D_XT})$.
	
	\begin{prop}\label{pro:operator_by_Y}
		For all Parseval frame $f=(f_i)_{i=1}^m$, we have
		\begin{equation*}
		\frac{1}{n}\sum_{i=1}^m p_i\nrm{f_i}^2\leq	\E\left(\nrm{T^*D_XT}\right).
		\end{equation*}
	\end{prop}
	\begin{proof}
	Pick a Parseval frame $f=(f_i)_{i=1}^m$. Following  \cite{tarazaga1991more}, we have 
	$$\nrm{T^*D_XT}\geq \frac{1}{n}\sum_{i=1}^m X_i\nrm{f_i}^2,$$ and the result follows.		
	\end{proof}

	The next proposition shows that no Parseval frames, can enhance the mean of the random error operator under  the threshold $p_n$.
	\begin{prop}\label{pro:estim_Expect_1}Fix a Parseval frame $f=(f_i)_{i=1}^m$. Then
		 $$\E\left(\nrm{T^*D_XT}\right)\geq p_{n}.$$
	\end{prop}
	\begin{proof}
		Since $f$ is Parseval, then for all $k=1,\dots,m$ and $z\in \B$ we may write
		\begin{equation}\label{eq:sum_by_parseval}
			\sum_{i=1}^mX_i \scal{f_i,z}^2=\sum_{i\ne k}^m\left(X_i-X_k\right) \scal{f_i,z}^2+X_k\nrm{z}^2
		\end{equation}
		Thus, for all $k=1,\dots,m$ and all $z\in\B$
		\begin{equation}\label{eq:bounds_oper}
			X_k\nrm{z}^2+\sum_{i\ne k}^m\left(X_i-X_k\right) \scal{f_i,z}^2\leq \nrm{T^*D_XT}.
		\end{equation}
		Letting $k=n$ in \eqref{eq:bounds_oper} gives
		$$\E\left(\nrm{T^*D_XT}\right)\geq \sup_{\nrm{z}\leq 1}\left(\sum_{i\ne n}^m\left(p_i-p_n\right) \scal{f_i,z}^2+p_n\nrm{z}^2\right).$$
		Keeping in mind that the $p_1\leq p_2\leq \dots \leq p_m$, we get
		$$\E\left(\nrm{T^*D_XT}\right)\geq \sup_{\nrm{z}\leq 1}\left(\sum_{i=1}^{n-1}\left(p_i-p_{n}\right) \scal{f_i,z}^2+p_{n}\nrm{z}^2\right).$$
		Now since $\mr{Span}(f_i)_{i=1}^{n-1}\subsetneqq \Rn$, there exists $z\in\Rn$ with $\nrm{z}=1$ and $\scal{f_i,z}=0$ for all $i=1,\dots,n-1$ and this proves our claim.
	\end{proof}
	
	\begin{rem}\label{rem:on_span_f}
		In the left side of the inequality in Proposition \ref{pro:estim_Expect_1}, we may replace $p_n$ by $p_{k}$ where 
		$$k=\max\Big\{j\in\set{2,\dots,m},\mr{Span}(f_i)_{i=1}^{j-1}
		\subsetneqq \Rn\Big\}.$$
	\end{rem}
	
	The remaining of this paper is devoted to study the minimization problem related to the worst  one-erasure case.
	
	\section{Worst r-Erasure case}\label{sec:worst_r_erasure}
	For a $(m,n)$-frame $f:=(f_i)_{i=1}^m$ and $r=1,\dots,m$, define
	\begin{equation}\label{eq:def_d_p_r}
		d_{p,r}(f) := \max_{J\in \mc{J}_r}\nrm{(f_{i})_{i\in J}(f_{i})_{i\in J}^*}\prod_{i\in J}p_i\prod_{i\notin J}(1-p_i),
	\end{equation}
	where  $\mc{J}_r$ stands the set of none empty subsets  of $\set{1,\dots,m}$ having exactly $r$ elements, $r\geq 1$, and  $\left(f_{i}\right)_{i\in J}$ denotes the matrix whose columns are $f_{i}$, $i\in J$. 
	
	\begin{rem}\label{rem:expect_and_worst_case}
		If we set  $N:=X_1+\dots+X_m$, then since the $X_i$'s are independent, straightforward calculations give us
		\begin{align*}
			\mathbb{E}\left(\nrm{T^*D_XT}\right)&=\sum_{r=1}^{m}\mathbb{P}(N=r)	\mathbb{E}\left(\nrm{T^*D_XT}|N=r \right)\\
			&=\sum_{r=1}^{m}\sum_{J\in \mc{J}_r} \nrm{\left(f_{i}\right)_{i\in J} \left(f_{i}\right)_{i\in J}^*}\prod_{i\in J}p_{i} \prod_{i\notin J}(1-p_i).
		\end{align*}
		By consequence,
		\begin{equation}\label{eq:comp_expect_r_worst_case}
			\mathbb{E}\left(\nrm{T^*D_XT}\right)
			\leq\sum_{r=1}^{m} {r\choose m}d_p,r(f).
		\end{equation}
	\end{rem}
	
	\medskip
	
	It is easily checked that the value
	\begin{equation}\label{eq:def_e_1_p}
		e_{p,1}(m,n):= \inf_{f\in\mathcal{P}(m,n)}d_{p,1}(f),
	\end{equation}
	is attained. This allows us to define 
	\begin{equation}\label{eq:def_E_1_p}
		\mathcal{E}_{p,1}(m,n):= \lbrace f\in \mathcal{F}(m,n) : d_{p,1}(f) = e_1(m,n)\rbrace.
	\end{equation}
	
	Proceeding inductively, we define 
	\begin{equation}\label{eq:def_e_r_p}
		e_{p,r}(m,n):= \inf_{f\in\mathcal{E}_{r-1}(m,n)}d_{p,r}(f), \quad 2\leq r \leq m,
	\end{equation}
	and
	\begin{equation}\label{eq:def_E_r_p}
		\mathcal{E}_{p,r}(m,n):= \lbrace f\in \mathcal{E}_{p,r-1}(m,n): d_{p,r}(f) = e_{p,r}(m,n)\rbrace.
	\end{equation} 
	
	Concretely, by investigating $e_{p,r}$ and $	\mathcal{E}_{p,r}$, we are looking for optimizing the worst   $r$-erasure case
	
	\begin{rem}\label{rem:assumption_on_p}
		Having a probability $p_i=0$ means that no erasure occurs at the $i^{\text{th}}$ channel, which is almost impossible in real situations. Alternatively, having a probability $p_i=1$ means that the data transmitted via the $i^\text{th}$ channel gets lost with certainty and then this channel is useless. 
		
		Therefore, with these considerations in mind, from now on we will assume that $0< p_i<1$ for all $i=1,\dots,m$.
	\end{rem}
	
	Before going further in our developments, we point out here   that in the case of equal probability, that is  $0< p_1=\dots= p_m=p<1$, the models  studying optimal frames for erasures in  \cite{casazza2003equal,holmes2004optimal,bodmann2005} are a particular case of our probabilistic model. Indeed, in this setting we get 
	$\mathcal{E}_{p,r}(m,n)=\mathcal{E}_r(m,n)$ for all $r=1,\dots,m,$ where 
	$\mathcal{E}_r(m,n)$ is the set of optimal $(m,n)-$Parseval frames for $r-$erasures according to the models in \cite{casazza2003equal,holmes2004optimal,bodmann2005}. 
	\medskip
	
	Now we will proceed to prepare the needed tools for our main result Theorem  \ref{thm:One_earsure_optimums}.
	\begin{defn}\label{def:condition_H}
		Let $\alpha:=(\alpha_1,\dots,\alpha_m)$ be a sequence of $m$ positive real numbers. The sequence  $\alpha$ is said  to satisfy the condition \eqref{eq:H} if we have
		\begin{equation}\label{eq:H}\tag{H}
			\alpha_i\geq \frac{n}{\sum_{k=1}^m\frac{1}{\alpha_k}}, \quad \forall i=1,\dots,m.
		\end{equation} 
	\end{defn}
	
	The following lemma is crucial for the characterization of optimal frames in the 1-erasure case in  our model.
	\begin{lem}\label{lem:alpha_min_pb}
		Let $h$, $\alpha_1,\dots,\alpha_m$ be positive   real numbers. Then the minimization problem
		
		\begin{equation}\label{eq:min_pb}
			\min_{{}^{t_1+\dots+t_m=h}_{t_i\geq 0,\, i=1,\dots,m}}\max\lbrace\alpha_1t_1,\dots,\alpha_mt_m\rbrace,
		\end{equation}
		admits a unique solution given by 
		\begin{equation}\label{eq:expr_min_solution}
			\frac{h}{\sum_{i=1}^m\frac{1}{\alpha_i}}\left(\frac{1}{\alpha_1},\dots,\frac{1}{\alpha_m}\right).
		\end{equation}Consequently, its optimal value is $\frac{h}{\sum_{i=1}^m\frac{1}{\alpha_i}}.$
	\end{lem}
	\begin{proof}
		First of all, by a compacity argument, the minimization problem \eqref{eq:min_pb} does have a solution, say $(w_1,\dots,w_m)$. Since the $m$-tuple given by \eqref{eq:expr_min_solution} satisfies the conditions of the minimization problem \eqref{eq:min_pb}, we get 
		$$\alpha_k w_k \leq \frac{h}{\sum_{i=1}^m\frac{1}{\alpha_i}},\quad \forall\, k=1,\dots,m.$$
		
		If one among the $m$ inequalities above is strict, then this will lead to 
		$$\sum_{k=1}^m w_k<\sum_{k=1}^m\frac{h}{\alpha_k\sum_{i=1}^m\frac{1}{\alpha_i}}=h,$$
		which is a contradiction and ends the proof.
	\end{proof}
	
	\begin{rem}\label{rem:min_pb}
		It is worth noting that the minimization problem \eqref{eq:min_pb} equivalents to 
		$$\min_{{}^{t_1+\dots+t_m=h}_{t_i> 0,\, i=1,\dots,m}}\max\lbrace\alpha_1t_1,\dots,\alpha_mt_m\rbrace.$$
	\end{rem}
	
	From now on, we fix  $p:=(p_1,\dots,p_m)\in \left(]0,1[\right)^m$ and set $\tilde{p}:=(\tilde{p}_1, \dots,\tilde{p}_m)$ where
	\begin{equation}\label{eq:def_p_i_tielda}
		\tilde{p}_i=p_i\prod_{j\ne i}(1-p_j).
	\end{equation}
	
	For simplicity, we will set  $I:=\set{1,\dots,m}$.  
	
	Appealing to Proposition \ref{pro:fundam_ineq}, the minimization problem \eqref{eq:def_e_1_p} becomes
	\begin{equation}\label{eq:e_p_1_equiv_expr}
		e_{p,1}(m,n)=\min_{{}^{a_1+\dots+a_m=n}_{a_1,\dots,a_m\in [0,1]}} \left(\max_{i\in I }\tilde{p}_i a_i\right),
	\end{equation}
	
	As we will see, the optimal value  $e_{p,1}(m,n)$ depends on whether  $\tilde{p}$ satisfies the condition \eqref{eq:H} or not. The next lemma is essential in determining this value.

	\begin{lem}\label{lem:d_existence}
		Assume that the set $$J:=\Big\{ i\in I ,\  \tilde{p}_i \text{ do not satisfy the condition } \eqref{eq:H}\Big\}$$ is not empty. Then the none negative integer 
		$$d:=\max\set{j\in I , \tilde{p}_j\sum_{k=j+1}^{m}
			\frac{1}{\tilde{p}_k}<n-j},$$
		is well defined and satisfies $\mathrm{card}(J)\leq d\leq  n-1$ and 
		\begin{equation}\label{eq:d_def}
			\tilde{p}_{d}<\frac{ n -d}{\sum_{k=d+1}^m\frac{1}{\tilde{p}_k}}\leq \tilde{p}_{d+1}.
		\end{equation}
	\end{lem}
	\begin{proof} Recall that, following Remark \ref{rem:assumption_on_p}, we have  $0<\tilde{p}_1\leq...\leq \tilde{p}_{m}$ which yields $J=\lbrace 1,...,\mathrm{card}(J)\rbrace$. 	
		Set $$L:=\set{j\in I , \tilde{p}_j\sum_{k=j+1}^{m}
			\frac{1}{\tilde{p}_k}<n-j}.$$ 	
		
		By its definition, for all $j\in J$ we have
		\begin{equation}\label{eq:conseq_def_J}
			\tilde{p}_j\sum_{k\in I}\frac{1}{\tilde{p}_k}<  n.
		\end{equation}
		
		From \eqref{eq:conseq_def_J} we see that  
		$$\tilde{p}_j\sum_{k=j+1}^m\frac{1}{\tilde{p}_k}<  n-\tilde{p}_j\sum_{k=1}^{j}\frac{1}{\tilde{p}_k}\leq n-j,$$ for all $j\in J$. This shows that $J\subset L$ and then $d:=\max   L$ is  well defined $\mathrm{card}(J)\leq d\leq  n-1$. 	
		Finally, \eqref{eq:d_def} follows from the definition of $d$. 
	\end{proof}
	
	In  the sequel, we define the index of the distribution $p$ by  $i(p)=0$ when $\tilde{p}$ satisfies the condition \eqref{eq:H}, and $i(p)=d$ otherwise, where $d$ is given by Lemma \ref{lem:d_existence}.
	
	The value of $e_{p,1}(m,n)$ is given bellow.   
	\begin{prop}\label{pro:e_p_1_value}
		We have $$e_{p,1}(m,n)=\frac{ n -i(p)}{\sum_{j=i(p)+1}^m\frac{1}{\tilde{p}_j}},
		$$ where $i(p)$ is the index of the distribution $p$.
	\end{prop}
	\begin{proof} 
		If  $\tilde{p}$ satisfies the condition \eqref{eq:H}, then by definition we have  $i(p)=0$ and the result follows from   \eqref{eq:e_p_1_equiv_expr} and Lemma \ref{lem:alpha_min_pb}. 
		
		Now assume that $\tilde{p}$ does not satisfy the condition \eqref{eq:H} and, for simplicity, set $i(p)=\gamma$ and 
		$T:=\lbrace \gamma+1,\dots,m \rbrace$.
		Appealing to Lemma \ref{lem:alpha_min_pb}, we have
		\begin{align*}
			&\min_{{}_{0\leq a_i\leq 1,\ i=1,\dots,m} ^{\sum_{i=1}^ma_i=n}} \left(\max_{i\in I}\tilde{p}_ia_i\right)\geq  \min_{{}_{0\leq a_i\leq 1,\ i=1,\dots,m} ^{\sum_{i=1}^ma_i=n}} \left(\max_{ i\in T}\tilde{p}_ia_i\right) \\
			&=\min_{{}^{0<a_i\leq 1}_{i\notin T}} \left(\min_{{}_{0<a_i\leq 1 \ i\in T}^{\sum_{i\in T}a_i=n-\sum_{i\notin T}a_i}}\left(\max_{i\in T}\tilde{p}_ia_i\right)\right)\\
			&=\min_{{}^{0<a_i\leq 1}_{i\notin T}} \frac{n-\sum_{i\notin T}a_i}{\sum_{j\in T}\frac{1}{\tilde{p}_j}}\\
			&=\frac{ n -\gamma}{\sum_{j\in T}\frac{1}{\tilde{p}_j}}
		\end{align*}
		In other hand, by letting $a:=(a_1,\dots,a_m)$  such that  
		\begin{equation}\label{eq:expr_sol_optimal}
			\begin{cases}
				a_i=\frac{ n -\gamma}{\tilde{p}_i\sum_{k=\gamma+1}^{m }\frac{1}{\tilde{p}_k}} \quad  i\in T,\\
				a_i=1, \quad i\notin T,
			\end{cases}
		\end{equation} 
		then from \eqref{thm:One_earsure_optimums}, we see that $a_i\in [0,1]$ for all $i=1,\dots, m$ and actually the sequence $a$ is an optimal solution for the minimization problem \eqref{eq:e_p_1_equiv_expr}. Thus,  	
		$$
		\min_{f\in \mathcal{P}(m,n)} d_{p,1}(f)=\frac{ n -\gamma}{\sum_{j\in T}\frac{1}{\tilde{p}_j}},
		$$ which ends the proof.	
	\end{proof}
	
	We are now ready to establish our main result.
	\begin{thm}\label{thm:One_earsure_optimums}
		Let $f:=(f_i)_{i=1}^m$ be a Parseval frame in $\R^n$.  Then $f\in \mathcal{E}_{p,1}(m,n)$ if and only if $f$ satisfies 
		$$\begin{cases}
			\nrm{f_i}^2= \frac{ n -i(p)}{\tilde{p}_i\sum_{k=i(p)+1}^{m }\frac{1}{\tilde{p}_k}}, \quad&\text{if} \quad  i\geq i(p)+1,\\
			\nrm{f_i}= 1 \quad &\text{otherwise} ,\\
		\end{cases}$$
		where $i(p)$ is the index of the distribution $p$.
	\end{thm}
	\begin{proof} 	
		
		For simplicity, we will set $i(p)=\gamma$ and $T:=\set{\gamma+1,\dots,m}$.
		
		Appealing to Proposition \ref{pro:e_p_1_value}, we see that the sufficient condition of our claim is obvious.
		
		Now assume that $f\in \mathcal{E}_{p,1}(m,n)$ and set $\nrm{f_i}^2=a_i$. Then from \eqref{eq:def_e_1_p} and Proposition \ref{pro:e_p_1_value} we have 
		\begin{equation}\label{eq:optimal_pty_1}
			\max_{i\in I}\tilde{p}_ia_i=\frac{ n -\gamma}{\sum_{j\in T}\frac{1}{\tilde{p}_j}}\qquad \mathrm{and}\qquad\sum_{i=1}^ma_i=n.
		\end{equation}  
		Since $T\subset I$,  then  using Lemma \ref{lem:alpha_min_pb} we get 
		\begin{align*}
			\max_{i\in I}\tilde{p}_ia_i&\geq\max_{i\in T}\tilde{p}_ia_i\\
			&\geq  \min_{\substack{\sum_{i\in T}x_i=n-\sum_{i\notin T}a_i\\ x_i\geq0}}\max_{i\in I}\tilde{p}_ix_i\\
			&=\frac{n-\sum_{i\notin T}a_i}{\sum_{j\in T}\frac{1}{\tilde{p}_j}}.
		\end{align*}
		Therefore, according to \eqref{eq:optimal_pty_1} we get 
		$$\gamma\leq \sum_{i\notin T}a_i\leq d,$$ and then 
		$\sum_{i\notin T}a_i= \gamma$. Further, since $a_i\in [0,1]$ for all $i\in I$, we infer that 
		$$a_i=1\qtxtq{for all} i\notin T.$$
		Thus, from \eqref{eq:optimal_pty_1} we deduce that
		\begin{equation}\label{eq:optimal_pty_2}
			a_i\leq \frac{ n -\gamma}{\tilde{p}_i\sum_{j\in T}\frac{1}{\tilde{p}_j}}\qtxtq{and}\sum_{i\in T}a_i= n -\gamma,
		\end{equation}  
		for all $i\in T$, which leads to 
		$$a_i= \frac{ n -\gamma}{\tilde{p}_i\sum_{j\in T}\frac{1}{\tilde{p}_j}},$$
		for all $i\in T$ and this proves the necessary condition of our claim.
	\end{proof}
	
	A direct application of Theorem \ref{thm:One_earsure_optimums}, gives us more insights on optimums in the case of 2-erasures phenomena  and when $\tilde{p}$ satisfies the condition \eqref{eq:H}. 
	
	Indeed, in the case of 2-erasure, that is $r=2$, we are led to calculate, for $i,j\in \set{1,\dots,m}$, the operator norm $$\nrm{(f_{i},f_j)(f_{i},f_j)^*},$$
	where $(f_{i},f_j)$ denotes the matrix whose columns are $f_i$ and $f_j$ respectively. We notice that 
	$$\nrm{(f_{i},f_j)(f_{i},f_j)^*}=\nrm{(f_{i},f_j)^T(f_{i},f_j)}=\begin{bmatrix}
		\nrm{f_i}^2&  \scal{f_i,f_j}  \\  \scal{f_i,f_j}&  \nrm{f_j}^2 
	\end{bmatrix},
	$$
	from which we deduce that 
	\begin{equation}\label{eq:expr_d_p,2}
		d_{p,2}(f) = \frac{1}{2}\max_{\substack{i,j=1,\dots,m\\ i\neq j}}\tilde{p}_{ij}\Big(\nrm{f_i}^2+\Vert f_j\Vert^2+\sqrt{(\nrm{f_i}^2-\Vert f_j\Vert^2)^2+4(\langle f_i,f_j\rangle)^2} \Big)
	\end{equation}
	where 
	\begin{equation}\label{eq:def_p_ij_tielda}
		\tilde{p}_{ij}:=p_ip_j\prod_{k\neq i,j}(1-p_k).  
	\end{equation}

	\begin{prop}\label{pro:bounds_2_erasure_with_H}	
		Assume that $\tilde{p}$ satisfies the condition \eqref{eq:H}. Then for all $f\in\mc{E}_{p,1}(m,n)$ we have 
		\begin{equation}\label{eq:def_operator_B}\small d_{p,2}(f)
			=\frac{1}{2}\max_{i,j\in I } \left ( c \left(\frac{p_i}{1-p_i}+\frac{p_j}{1-p_j} +\right )+\sqrt{c^2\left (\frac{p_i}{1-p_i}- \frac{p_j}{1-p_j}\right) ^2+4\left(\tilde{p}_{i,j}\langle f_i,f_j\rangle\right)^2}\right), 
		\end{equation} 
		where $c:=\frac{ n }{\sum_{k=1}^m\frac{1}{\tilde{p}_k}}$. 	
	\end{prop}
	
	\begin{cor}\label{cor:d_p_2_equi_proba}Assume that $p_i=p$ for all $i\in I$ with  $0< p<1$. Then 
		$$e_{p,2}(m,n)\sim \frac{ n p^2}{m(1-p)}, \qtxtq{as} m\to\infty.$$
	\end{cor}
	\begin{proof}Since in this case $\tilde{p}$ satisfies the condition \eqref{eq:H}, then from Proposition \ref{pro:bounds_2_erasure_with_H},  for all  $f\in \mathcal{E}_{p,1}(m,n)$  we get that 
		\begin{align*}
			d_{p,2}(f)&=\frac{1}{2}\max_{i,j\in I } \left (\frac{ n p^2}{m(1-p)}+2p^2(1-p)^{m-2}\abs{\langle f_i,f_j\rangle}\right )\\
			&=\frac{ n p^2}{m(1-p)} + p^2(1-p)^{m-2}\max_{1\leq i\neq j\leq m}\left (\abs{\langle f_i,f_j\rangle}\right )
		\end{align*}
		
		Appealing to  Theorem \ref{thm:One_earsure_optimums},  we have
		$$\abs{\langle f_i,f_j\rangle}\leq \frac{ n }{m},$$ for all $f\in\mathcal{E}_{p,1}(m,n)$ and this proves our claim.	
	\end{proof}
	
	\section{Comparison with other models}\label{sec:comparison}
	In this section we reuse the notations of Section \ref{sec:revising_prob_model}. For a sequence $(X_i)_{i=1}^m$ of independent Bernoulli random variables associated to the $m$ transmission channels and its associated distribution $(p_i)_{i=1}^m\in ]0,1[^m$. 
	
	According to our model (RPM), the mean random error in the case of one-erasure related to a Parseval frame $(f_{i})_{i=1}^m\in\mc{P}(m,n)$ is modeled by
	\begin{equation}\label{eq:expect_one_erasure}
		\mathbb{E}\left(\nrm{T^*D_XT} | N=1\right)=\frac{1}{\sum_{i=1}^m \tilde{p}_i}\sum_{i=1}^m \tilde{p}_i\nrm{f_{i}}^2 ,
	\end{equation}
	where $T:=T_f$ and $N=X_1+\dots+X_m$.
	
	We are concerned here by comparing our model (RPM) to the models (CM) and (PM) by estimating, for each of the three models, the quantity \eqref{eq:expect_one_erasure} when applied to their associated optimal Parseval frames and the distribution  $(p_i)_{i=1}^m$ being given.
	
	Concretely, we fix the erasure distribution $(p_i)_{i=1}^m$ and we set 
	\begin{equation}\label{eq:expect_one_erasure_M}
		\mathbb{E}_M:=\frac{1}{\sum_{i=1}^m \tilde{p}_i}\sum_{i=1}^m \tilde{p}_i\nrm{f_{i}}^2 ,
	\end{equation}
	where $M$ stands for one of the three models (CM), (PM) and (RPM) and $(f_{i})_{i=1}^m$ is an optimal Parseval frame obtained for the model M. Namely,
	\begin{itemize}
		\item[-] Following  \cite{casazza2003equal,holmes2004optimal}, optimal Parseval frames $(f_i)_{i=1}^m$ for (CM) satisfy  $$\nrm{f_i}^2=\frac{n}{m}$$ and then \begin{equation}\label{eq:CM_expect}
			\mathbb{E}_{CM}=\frac{n}{m}.
		\end{equation}
		\item[-] Following \cite{li2018frame}, optimal Parseval frames $(f_i)_{i=1}^m$ for (PM) satisfy $$\nrm{f_i}^2=\frac{n}{m-1}\left(1-\frac{ p_i}{\sum_{i=1}^m p_i}\right)$$ which gives  
		\begin{equation}\label{eq:PM_expect}
			\mathbb{E}_{PM}=\frac{n}{(m-1)\sum_{i=1}^m \tilde{p}_i}\sum_{i=1}^m\tilde{p}_i\left(1-\frac{ p_i}{\sum_{i=1}^m p_i}\right).
		\end{equation}
		\item[-] Finally for (RPM), Theorem \ref{thm:One_earsure_optimums} says that optimal Parseval  frames $(f_i)_{i=1}^m$ satisfy $$\begin{cases}
			\nrm{f_i}^2= \frac{ n -i(p)}{\tilde{p}_i\sum_{k=i(p)+1}^{m }\frac{1}{\tilde{p}_k}}, \quad&\text{if} \quad  i\geq i(p)+1,\\
			\nrm{f_i}= 1 \quad &\text{otherwise} ,\\
		\end{cases}$$ and accordingly 
		\begin{equation}\label{eq:RPM_expect}
			\mathbb{E}_{RPM}=\frac{1}{\sum_{i=1}^m \tilde{p}_i}\left(\sum_{i=1}^d\tilde{p}_i+\frac{(m-i(p))(n-i(p))}{\sum_{k=i(p)+1}^m\frac{1}{\tilde{p}_k}}\right),
		\end{equation} where $i(p)$ is the index of the distribution $p$.
	\end{itemize}

	We will prove that $\mathbb{E}_{RPM}$ is the smallest compared to $\mathbb{E}_{PM}$ and $\mathbb{E}_{CM}$. 
	Further we will classify the three models in ascending order (See Propositions \ref{pro:comp_PM_CM} and \ref{pro:comp_RPM_PM}). 
	
	A useful ingredient to achieve these comparisons, is the following lemma  known as Chebyshev inequality, and for which we give a sketch of the proof for the convenience of the reader.
	\begin{lem}\label{lem:Chebyshev_ineq}
		Let $a:=(a_i)_{i=1}^m$ and $b:=(b_i)_{i=1}^m$ be two sequences of real numbers such that $a$ and $b$ are both increasing. Then we have $$\sum_{i=1}^{m}a_ib_i\geq \frac{1}{m} \left(\sum_{i=1}^{m}a_i\right) \left(\sum_{i=1}^{m}b_i\right).$$
		
		As a consequence, if $a$ is increasing and $b$ is decreasing then
		$$\sum_{i=1}^{m}a_ib_i\leq \frac{1}{m} \left(\sum_{i=1}^{m}a_i\right) \left(\sum_{i=1}^{m}b_i\right).$$  
	\end{lem}
	\begin{proof}
		In the case when $a$ and $b$ are both increasing, the result follows by expanding the sum
		$$A:=\sum_{j=1}^{m}\sum_{k=1}^{m}(a_{j}-a_{k})(b_{j}-b_{k}),$$ 
		and noting that $A\geq 0$.	
	\end{proof}
	
	Now we are ready to compare $\mathbb{E}_{PM}$ and $\mathbb{E}_{CM}$.
	\begin{prop}\label{pro:comp_PM_CM}
		Keeping the same notations as above, we have
		$$\mathbb{E}_{PM}\leq\mathbb{E}_{CM}.$$
	\end{prop}
	\begin{proof} 
		For simplicity, we set  $P=\sum_{i=1}^m p_i$. We get then
		\begin{align*}
			\mathbb{E}_{PM}-\mathbb{E}_{CM}
			&=\frac{n}{\sum_{i=1}^m \tilde{p}_i} \sum_{i=1}^m\tilde{p}_i\left[\frac{1}{m-1}\left(1-\frac{p_i}{P}\right)-\frac{1}{m}\right].
		\end{align*}
		Note that the sequence $(\tilde{p}_i)_{i=1}^m$ is increasing and  	
		$\left(\frac{1}{m-1}\left(1-\frac{p_i}{P}\right)-\frac{1}{m}\right)_{i=1}^m$ is decreasing, and then by Lemma \ref{lem:Chebyshev_ineq}   we get 
		$$\sum_{i=1}^m\tilde{p}_i\left[\frac{1}{m-1}\left(1-\frac{p_i}{P}\right)-\frac{1}{m}\right]\leq \frac{1}{m}\left(\sum_{i=1}^m\tilde{p}_i\right)\left(\sum_{i=1}^m\left(\frac{1}{m-1}\left(1-\frac{p_i}{P}\right)-\frac{1}{m}\right)\right)=0,$$ and this proves our claim.
	\end{proof}
	
	Now we establish a comparison between $\mathbb{E}_{PM}$ and $\mathbb{E}_{RPM}$.
	\begin{prop}\label{pro:comp_RPM_PM}
		Keeping our notations as above, we have  
		$$\mathbb{E}_{PM}-\mathbb{E}_{RPM}\geq \frac{1}{\sum_{i=1}^m\tilde{p}_i}\frac{i(p)(m-n)(m-n-1)}{n(m-i(p))(m-i(p)-1)}\left(\sum_{i=i(p)+1}^m\tilde{p}_i-\frac{\sum_{k=i(p)+1}^m \tilde{p}_k p_k}{\sum_{k=i(p)+1}^m p_k }\right),$$ where $i(p)$ is the index of the distribution $p$.
	\end{prop}
	\begin{proof}
		For simplicity we set $i(p)=d$,  $P:=\sum_{i=1}^m p_i$ and $$A:=\left(\sum_{i=1}^m\tilde{p}_i\right)(\mathbb{E}_{RPM}-\mathbb{E}_{PM}).$$
		We have  \begin{align*}
			&A=\sum_{i=1}^d \tilde{p}_i +\frac{(m-d)(n-d)}{\sum_{k=d+1}^m\frac{1}
				{\tilde{p}_k }}-\frac{n}{m-1}\sum_{i=1}^m\tilde{p}_i \left(1-\frac{p_i}{P}\right)\\
			&= \left(1-\frac{n}{m-1}\right)\sum_{k=1}^d\tilde{p}_k  +\frac{(m-d)(n-d)}{\sum_{k=d+1}^m\frac{1}	{\tilde{p}_k }}
			-\frac{n}{m-1}\left(\sum_{i=d+1}^m \tilde{p}_i
			- \frac{\sum_{k=1}^m\tilde{p}_k p_k}{P}\right)
		\end{align*}
		
		Since $(p_i)_{i=1}^m$ is increasing, then from \eqref{eq:d_def} we infer that 
		\begin{equation}\label{eq:pty_d}
			\sum_{i=1}^d \tilde{p}_i\leq \frac{d(n-d)}{\sum_{i=d+1}^m \frac{1}{\tilde{p}_i}}
		\end{equation}
		which leads to
		\begin{align*}
			&A\leq 
			(n-d)\frac{d(m-n-1)+(m-1)(m-d)}{(m-1)\sum_{i=d+1}^m\frac{1}{\tilde{p}_i}}
			-\frac{n}{m-1}\left(\sum_{i=d+1}^m\tilde{p}_i- \frac{\sum_{k=1}^m \tilde{p}_k p_k}{P}\right)
		\end{align*}
		Alternatively, by setting $\pi_p:=\prod_{i=1}^{m}p_i$  and $P_{d,m}:=\sum_{k=d+1}^m p_k$ and $$g(x):= \frac{x}{1-x}
		\left(1-\frac{x}{P_{d,m}}\right),$$ we see that
		\begin{align*}
			&\tilde{p}_i
			\left(1-\frac{p_i}{\sum_{k=d+1}^m p_k}\right)=\pi_p \frac{p_i}{1-p_i}
			\left(1-\frac{p_i}{P_{d,m}}\right)=\pi_p g(p_i)
		\end{align*}
		and straightforward calculations show that  
		$$g(x)-g(y)= \frac{(x-y)(P_{d,m}-x-y+xy)}{(1-x)(1-y)P_{d,m}}.$$ This shows in particular that the sequence $\left(\tilde{p}_i
		\left(1-\frac{p_i}{\sum_{k=d+1}^m p_k}\right)\right)_{i=d+1}^m$ is  increasing, and appealing to Lemma \ref{lem:Chebyshev_ineq} we get 
		\begin{equation}\label{eq:main_ineq}
			\sum_{i=d+1}^m\tilde{p}_i
			\left(1-\frac{p_i}{\sum_{k=d+1}^m p_k}\right)\geq \frac{(m-d)(m-d-1)}{\sum_{i=d+1}^m\frac{1}{\tilde{p}_i}}.
		\end{equation}
		By consequence, we get
		\begin{align*}
			&\frac{(m-1)}{n}A\leq 
			(n-d)\frac{d(m-n-1)+(m-1)(m-d)}{n(m-d)(m-d-1)}	
			\left(\sum_{i=d+1}^m\tilde{p}_i-\frac{\sum_{i=d+1}^m \tilde{p}_i p_i}{\sum_{i=d+1}^m p_i }\right)\\
			&\qquad 
			-\left(\sum_{i=d+1}^m\tilde{p}_i-\frac{\sum_{i=1}^m \tilde{p}_i p_i}{P}\right)\\
			&\leq \frac{\sum_{i=1}^m \tilde{p}_ip_i}{P}-\frac{\sum_{i=d+1}^m \tilde{p}_ip_i}{\sum_{i=d+1}^m p_i }-\frac{d(m-n)(m-n-1)}{n(m-d)(m-d-1)}\left(\sum_{i=d+1}^m\tilde{p}_i-\frac{\sum_{i=d+1}^m \tilde{p}_i p_i}{\sum_{i=d+1}^m p_i }\right)
		\end{align*}
		Now, from \eqref{eq:d_def} we see that  for all $i=1,\dots,d$ and $j=d+1,\dots,m$ we have  
		$$\tilde{p}_ip_i\leq \frac{n-d}{\sum_{k=d+1}^m\frac{1}{\tilde{p}_k }} p_i\leq \frac{n-d}{\sum_{k=d+1}^m\frac{1}{\tilde{p}_k }} p_j\leq\tilde{p}_j p_j,$$ and then summing over $i$, and $j$ yields 
		$$\frac{1}{\sum_{i=1}^dp_i}\sum_{i=1}^d
		\tilde{p}_ip_i\leq \frac{1}{\sum_{j=d+1}^m p_j}\sum_{j=d+1}^m\tilde{p}_jp_j,$$ from which we deduce that
		$$\frac{1}{P}\sum_{i=1}^m
		\tilde{p}_ip_i\leq \frac{1}{\sum_{j=d+1}^m p_j}\sum_{j=d+1}^m\tilde{p}_jp_j,$$ and this ends the proof.
	\end{proof}
	
	As a direct consequence of Proposition \ref{pro:comp_RPM_PM} we get
	\begin{cor}\label{cor:comp_RPM_PM}
		Keeping the same notations, we have $$\mathbb{E}_{PM}-\mathbb{E}_{RPM}\geq 
		\frac{i(p)(m-n)(m-n-1)}{n(m-i(p))(m-i(p)-1)}
		\frac{\sum_{i=i(p)+1}^m\tilde{p}_i\left(1-\frac{1}{m-i+1}\right)}{\sum_{i=1}^m\tilde{p}_i}.$$
	\end{cor}
	\begin{proof}
		Note that for all $i\in \set{i(p)+1,\dots,m}$ we have
		$$\sum_{k=i(p)+1}^{m}p_k\geq (m-i+1)p_i,$$ and then 
		$$\frac{p_i}{\sum_{k=i(p)+1}^{m}p_k}\leq \frac{1}{m-i+1},$$ and from Proposition \ref{pro:comp_RPM_PM} we deduce that 
		\begin{align*}
			&\mathbb{E}_{PM}-\mathbb{E}_{RPM}\geq 
			\frac{i(p)(m-n)(m-n-1)}{n(m-i(p))(m-i(p)-1)}
			\frac{\sum_{i=i(p)+1}^m\tilde{p}_i\left(1-\frac{1}{m-i+1}\right)}{\sum_{i=1}^m\tilde{p}_i}
		\end{align*} 
		which is the desired result.	
	\end{proof}

We end this section by giving, based on the results established above, a concrete illustration of the  performance of our model compared to CM and PM models in case of 1-erasure.
\begin{example}\label{example1}
	Let $n=2,\ m=3$ and the probability of erasure for each of the 3 channels is $p=(\frac{1}{20},\ \frac{1}{20},\ \frac{5}{20})$. In this setting $i(p)=0$ and by the expression of $\mathbb{E}_M$ for each of the three models we get
	\begin{equation}\label{ex:expt_cond}
		\mathbb{E}_{RPM}=\frac{342}{1025},\qquad\quad
		\mathbb{E}_{PM}=\frac{74}{175},\qquad\quad
		\mathbb{E}_{CM}=\frac{2}{3}.
	\end{equation}
	From \eqref{ex:expt_cond}, we observe that on average under 1-erasure our model is $21.09\%$ better than PM model, and $49.95\%$ better than CM model.
\end{example}

	\bibliographystyle{abbrv}
	\bibliography{erasures-biblio.bib}
\end{document}